\documentclass[letterpaper, 10 pt, conference]{ieeeconf}  

\IEEEoverridecommandlockouts                              

\usepackage{afterpage}
\usepackage{epsfig} 
\usepackage{times} 
\usepackage{amsmath} 
\usepackage{amssymb}  
\usepackage{amsfonts}
\usepackage{mathptmx} 
\usepackage{tabularx}
\usepackage{epstopdf}
\usepackage{subfigure}
\usepackage{float}
\usepackage{epic,color}
\usepackage{mathrsfs}
\usepackage{dsfont,MnSymbol}
\usepackage{algorithm}
\usepackage{algorithmic}
\usepackage{multirow} 
\usepackage[T1]{fontenc}
\newcounter{rmnum}
\newenvironment{romannum}{\begin{list}{{\upshape (\roman{rmnum})}}{\usecounter{rmnum}
\setlength{\leftmargin}{14pt}
\setlength{\rightmargin}{8pt}
\setlength{\itemsep}{2pt}
\setlength{\itemindent}{-1pt}
}}{\end{list}}

\newcounter{anum}

\newlength{\noteWidth}
\long\def\notes#1{\ifinner
             {\tiny #1}
             \else
              \marginpar{\parbox[t]{\noteWidth}{\raggedright\tiny #1}}
               \fi}

\def\IEEEQEDclosed{\mbox{\rule[0pt]{1.3ex}{1.3ex}}}

\def\qed{\ifmmode\IEEEQEDclosed\else{\unskip\nobreak\hfil
\penalty50\hskip1em\null\nobreak\hfil\IEEEQEDclosed
\parfillskip=0pt\finalhyphendemerits=0\endgraf}\fi}

\def\qed{\hspace*{\fill}~\IEEEQED\par\endtrivlist\unskip}

\def\Re{\mathbb{R}}

\def\Sec#1{Sec.~\ref{#1}}

\def\notes#1{\marginpar{\tiny #1}\typeout{Notes!
Notes!
Notes!
}}
\renewcommand{\notes}[1]{\typeout{notes!}}
\def\FRAC#1#2#3{\genfrac{}{}{}{#1}{#2}{#3}}
\def\half{{\mathchoice{\FRAC{1}{1}{2}}%
{\FRAC{2}{1}{2}}%
{\FRAC{3}{1}{2}}%
{\FRAC{4}{1}{2}}}}

\def\Re{\field{R}}
\def\k{{\sf K}}

\def\Sec#1{Sec.~\ref{#1}}

\def\clZ{{\cal Z}}

\def\Sec#1{Sec~\ref{#1}}

\def\Expect{{\sf E}}

\def\Expect{{\sf E}}

\def\Sec#1{Sec.~\ref{#1}}

\def\P{{\sf P}}

\def\IEEEQEDclosed{\mbox{\rule[0pt]{1.3ex}{1.3ex}}}
\def\qed{\nobreak\hfill\IEEEQEDclosed}

\def\clZ{{\cal Z}}

\newtheorem{theorem}{Theorem}

\newtheorem{remark}{Remark}
\newtheorem{proposition}{Proposition}
\newtheorem{corollary}{Corollary}

\def\beq{\begin{eqnarray}} 
\def\bc{\begin{center}} 
\def\be{\begin{enumerate}}
\def\bi{\begin{itemize}} 
\def\bs{\begin{small}}
\def\bS{\begin{slide}}
\def\ec{\end{center}} 
\def\ee{\end{enumerate}}
\def\ei{\end{itemize}}
\def\es{\end{small}}
\def\eS{\end{slide}}
\def\eeq{\end{eqnarray}}

\newcommand{\newP}[1]{\noindent{\bf #1:}}

\newcommand{\trace}{\text{Tr}}

\newcommand{\PP}{{\sf P}}

\newcommand{\ud}{\,\mathrm{d}}

\def\Re{\mathbb{R}}

\def\Sec#1{Sec.~\ref{#1}}

\def\Expect{{\sf E}}

\def\clZ{{\cal Z}}

\renewcommand{\Re}{\mathbb{R}}

\def\FRAC#1#2#3{\genfrac{}{}{}{#1}{#2}{#3}}

\newcommand{\var}{\text{Var}}
\newcommand{\X}{X}
\newcommand{\Om}{\Omega}
\newcommand{\NN}{\mathcal{N}}

\newcommand{\mN}{m^{(N)}}
\newcommand{\SigN}{\Sigma^{(N)}}

\newcommand{\snorm}[1]{\|#1\|_s}
\newcommand{\Fnorm}[1]{\|#1\|_F}

\title{\LARGE \bf
Error Analysis for the Linear Feedback Particle Filter}

\author{Amirhossein Taghvaei and Prashant G. Mehta 
\thanks{Financial support from the NSF CMMI grant 1462773 is gratefully acknowledged. 
}
\thanks{A.~Taghvaei and P.~G.~Mehta are with the Coordinated
  Science Laboratory and the Department of Mechanical Science and
  Engineering at the University of Illinois at Urbana-Champaign (UIUC).}
}

\begin{document}

\maketitle
\thispagestyle{empty}
\pagestyle{empty}

\begin{abstract}
This paper is concerned with the convergence and the error analysis for the
feedback particle filter (FPF) algorithm.  The FPF is a controlled
interacting particle system where the control law is designed to solve
the nonlinear filtering problem.  For the linear Gaussian case,
certain simplifications arise whereby the linear FPF reduces to one
form of the ensemble Kalman filter.  For this and for the
more general nonlinear non-Gaussian case, it has been an open problem
to relate the convergence and error properties of the finite-$N$
algorithm to the mean-field limit (where the exactness results have
been obtained).  In this paper, the equations for empirical
mean and covariance are derived for the finite-$N$ linear FPF.
Remarkably, for a certain deterministic form of  FPF, the equations for
mean and variance are identical to the Kalman filter.  This allows
strong conclusions on convergence and error properties based on the classical filter stability
theory for the Kalman filter.  It is shown that the error converges to
zero {\em even} with finite number of particles. The paper also presents propagation of
chaos estimates for the finite-$N$ linear filter.  The error estimates
are illustrated with numerical experiments.  


\end{abstract}
\section{Introduction}

In recent years, there
has been a burgeoning interest in application of ideas and techniques from
statistical mechanics to control theory and signal processing.
Although some of these applications are classical (see e.g., Del
Moral~\cite{delmoralbook}), the recent impetus comes from explosive
interest in mean-field games, starting with two papers from 2007:
Lasry and Lions paper titled ``Mean-field games''~\cite{lasry07mean}
and a paper by Huang, Caines and Malham{\'e}~\cite{huacaimal07}.
These papers spurred interest in the analysis and synthesis of {\em
  controlled interacting particle systems}. 

Feedback particle filter (FPF) is an example of a controlled
interacting particle system to approximate the solution of the
continuous-time nonlinear filtering problem.  In FPF, the importance
sampling step of the conventional particle filter is
replaced with feedback control. Other steps
such as resampling, reproduction, death or birth of particles are
altogether avoided. 
 
The first interacting particle representation of the continuous-time filtering
problem appeared in the work of Crisan and Xiong~\cite{crisan10}.
Also in continuous-time settings, Reich and collaborators have derived
certain deterministic forms of the ensemble Kalman
filter~\cite{reich11,Reich-ensemble}.  These forms are identical to the linear FPF.  
An expository review of the continuous-time filters including the
progression from the Kalman filter (1960s) to the ensemble Kalman
filter (1990s) to the feedback particle filter (2010s) appears
in~\cite{TaghvaeiASME2017}.  In
discrete-time settings, Daum and collaborators have pioneered the
development of closely related particle flow
algorithms~\cite{DaumHuang08,daum2017generalized}.  

In numerical evaluations and comparisons, it is often found that
the control-based algorithms exhibit smaller simulation variance and
better scaling properties with the problem dimension (number of state
variables).  For example, several research groups have
reported favorable comparisons for the FPF algorithm as compared to
the traditional particle filter algorithms;
cf.,~\cite{stano2014,stano2013nonlinear,berntorp2015,surace_SIAM_Review,adamtilton_fusion13}.  
However, there is no theoretical justification/understanding of
this.  Much of the work for FPF and more broadly for the particle flow
algorithms and the mean-field game models has focused on the
properties of the mean-field limit (e.g., the exactness of the FPF has
been shown for the mean-field limit).

For the nonlinear FPF, the convergence analysis is difficult in part
because the gain function is implicitly defined as the solution
of a certain partial differential equation (pde) referred to as the
Poisson equation; cf.,~\cite{poisson15}.  For the linear Gaussian
case, the pde admits an explicit solution whereby the gain function is
the Kalman gain and the resulting linear FPF is an ensemble Kalman
filter~\cite{TaghvaeiASME2017}.  In this paper, two forms of linear
FPF are studied:   


\noindent
{\bf (A) Stochastic linear FPF:} The original formulation of the FPF
algorithm in the linear Gaussian setting~\cite[Eq.~(26)]{yang2016}.

\noindent
{\bf (B) Deterministic linear FPF} The optimal transport formulation
of the linear Gaussian FPF~\cite[Eq.~(15)]{AmirACC2016}. 

Both the formulations are exact in the following sense: In the
mean-field limit the distribution of the particles 
equals the posterior distribution of the filter.  The main difference
between the two formulations is that the process noise term in the
stochastic FPF is replaced with a deterministic term in the
deterministic FPF.  

The goal of this paper is to characterize the error properties of the
FPF in the limit when the number of particles $N$ is large but finite.  The
error metrics of interest include the mean-squared error between the
finite-$N$ estimates (empirical mean and the empirical covariance) and
their mean-field limits (conditional mean and covariance).
Additionally, it is of interest to investigate the convergence of
the empirical distribution of the interacting particle system to the
conditional distribution obtained in the mean-field limit.    

\newP{ Contributions of this paper}  
The evolution equations for the empirical mean and covariance are
derived for the two systems (A) and (B).  It is shown that these equations
for the deterministic FPF are identical to the Kalman filter.  The
evolution equations for the stochastic FPF include additional
stochastic terms due to process noise.  In the large $N$ limit, these
terms scale as $O(N^{-\frac{1}{2}})$.  For the deterministic FPF, the
following results are obtained in Prop.~\ref{prop:conv_error}: (i)
almost sure convergence as $t\rightarrow\infty$; (ii) mean-squared
convergence where the error is shown to convergence to zero as
$O(N^{-\frac{1}{2}})$.  Certain preliminary results on a propagation
of chaos analysis for the scalar problem appear in
Prop.~\ref{prop:prop-chaos}.



Closely related to this paper is the recent literature on stability
and convergence of the ensemble Kalman filter algorithms for both
linear and nonlinear problems.  Examples of the former
include~\cite{gland2009,mandel2015} in discrete-time setting
and~\cite{delmoral2016stability} in continuous-time setting.  Examples
of the latter
include~\cite{jana2016stability,delmoral2017stability,stuart2014stability}. 
More generally, the propagation
  of chaos analysis of interacting particle system
models has a rich history; cf.,~\cite{mckean66class,sznitman1991,rachev1998}.

\noindent{ \bf Notation}
$\NN(m,\Sigma)$ is a Gaussian probability distribution with mean
$m$ and covariance $\Sigma \succ 0$ ($\Sigma \succ 0$ means that the
matrix $\Sigma$ is positive definite).  For a vector $m$, $|m|$
denotes the Euclidean norm.  For a square matrix $\Sigma$,
$\Fnorm{\Sigma}$
denotes the Frobenius norm,
$\snorm{\Sigma}$ 
is the spectral norm,
$\Sigma^\top$ is the matrix-transpose, $\trace(\Sigma)$ is the
matrix-trace, and $\text{Ker}(\Sigma)$ denotes the null-space. 
There are three types of stochastic process
considered in this paper: (i) $X_t$ denotes the state of the (hidden)
signal at
time $t$; (ii) ${X}_t^i$ denotes the state of the
$i^{\text{th}}$ particle in a population of $N$ particles; and (iii)
$\bar{X}_t$ denotes the state of the McKean-Vlasov model obtained in
the mean-field limit ($N=\infty$).  The mean and the covariance for
these are denoted as follows: (i) ($m_t,\Sigma_t$) is the conditional mean and
the conditional covariance pair for $X_t$; (ii) ($\mN_t,\SigN_t$) is the
empirical mean and the empirical covariance for the ensemble
$\{X_t^i\}_{i=1}^N$; and (iii) ($\bar{m}_t,\bar{\Sigma}_t$) is the
conditional mean and the conditional covariance for $\bar{X}_t$.  
The
notation is tabulated in the accompanying
Table~\ref{tab:symbols-states}. 

 
 \begin{table}[h]
 \centering
 \begin{tabular}{|c|c|c|}
 \hline
 Variable & Notation & Equation \\ \hline\hline 
 State of the hidden process & $X_t$ & Eq.~\eqref{eq:dyn}\\
State of the $i^{\text{th}}$ particle in finite-$N$ sys.& $X_t^i$
&Eq.~\eqref{eq:Xit-s},~\eqref{eq:Xit-d} 
 \\
State of the McKean-Vlasov model  & $\bar{X}_t$ & Eq.~\eqref{eq:Xbart-stochastic},~\eqref{eq:Xbart} \\ 
\hline
  \vspace*{-0.1in}\\
  Kalman filter mean and covariance & ${m}_t,{\Sigma}_t$ & Eq.~\eqref{eq:KF-mean}-\eqref{eq:KF-variance} 
   \\
 Empirical mean and covariance & $\mN_t,\SigN_t$ & Eq.~\eqref{eq:empr_app_mean_var}
 \\
 Mean-field mean and covariance & $\bar{m}_t,\bar{\Sigma}_t$ & Eq.~\eqref{eq:Xbart-stochastic}-\eqref{eq:Xbart} 
 \\ \hline
 \end{tabular}
 \label{tab:symbols-states}
 \end{table}


The outline of the remainder of this paper is as follows:
Sec.~\ref{sec:prelim} introduces the two models of the FPF studied in
this paper.  Sec.~\ref{sec:mean_covariance} presents the results on
convergence and error estimates for the empirical mean and
covariance.  Sec.~\ref{sec:poa} presents the propagation of chaos
analysis.  All the proofs appear in the Appendix.      

\section{Problem formulation and background}
\label{sec:prelim}

Consider the linear Gaussian filtering problem:
\begin{subequations}
\begin{align}
\ud X_t &=  AX_t \ud t + \sigma_B\ud B_t\label{eq:dyn}\\
\ud Z_t &= CX_t\ud t + \ud W_t\label{eq:obs}
\end{align}
\end{subequations}
where $X_t \in \Re^d$ is the (hidden) state at time $t$, $Z_t \in \Re^m$ is the
observation; $A$, $C$, $\sigma_B$ are matrices of appropriate
dimension; and $\{B_t\}$, $\{W_t\}$ are mutually independent Wiener
processes taking values in $\Re^d$ and $\Re^m$, respectively. Without
loss of generality, the covariance matrices associated with $\{B_t\}$
and $\{W_t\}$ are identity matrices.
The initial condition $X_0$ is drawn from a Gaussian
distribution $\mathcal{N}(m_0,\Sigma_0)$, independent of $\{B_t\}$ and
$\{W_t\}$. The filtering problem is to compute the posterior
distribution $\P(X_t|\clZ_t)$ where $\clZ_t:=\sigma(Z_s;s\in[0,t])$
denotes the time-history of observations up to time $t$ (filtration).

For the linear Gaussian problem~\eqref{eq:dyn}-\eqref{eq:obs}, the posterior distribution
$\PP(\X_t|\clZ_t)$ is Gaussian $\mathcal{N}(m_t,\Sigma_t)$, whose mean
and covariance are given by the Kalman-Bucy filter~\cite{kalman-bucy}:
\begin{subequations}
\begin{align}
\ud m_t &= A m_t\ud t + \mathsf{K}_t(\ud Z_t - Cm_t\ud t)\label{eq:KF-mean}\\
\frac{\ud \Sigma_t}{\ud t} &= A \Sigma_t +  \Sigma_tA^\top + \sigma_B\sigma_B^{\top} - \Sigma_tC^\top C\Sigma_t\label{eq:KF-variance}
\end{align}
\end{subequations}
where $\k_t:=\Sigma_tC^\top$ is the Kalman gain and the filter is
initialized with the prior $(m_0,\Sigma_0)$. 


Feedback particle filter (FPF) is a controlled interacting
particle system to approximate the Kalman filter\footnote{Although the
considerations of this paper are limited to the linear Gaussian
problem~\eqref{eq:dyn}-\eqref{eq:obs}, the FPF algorithm is more
broadly applicable to nonlinear non-Gaussian filtering
problems~\cite{taoyang_TAC12,yang2016}.}.  In a numerical
implementation, the filter is simulated with $N$ interacting
particles where $N$ is typically large.  The analysis of the filter
is based on the so-called mean-field models which are obtained upon
replacing the interaction terms by their mean-field limits.  The
mean-field model is referred to as the McKean-Vlasov
model~\cite{mckean1966}.

We begin by presenting the
McKean-Vlasov stochastic differential equation (sde) for the linear FPF algorithm.  For these models,
the state at time $t$ is denoted as $\bar{X}_t$. Two types of FPF
algorithm are considered: (A) FPF using the constant gain
approximation of the gain function (Eq.~(26) in~\cite{yang2016}); and
(B) FPF obtained using optimal transportation (Eq.~(15) in~\cite{AmirACC2016}).
These algorithms are referred to as the stochastic linear FPF and the
deterministic linear FPF, respectively.

\newP{ (A) Stochastic linear FPF}  The state $\bar{X}_t$ evolves
according to the McKean-Vlasov sde:
\begin{align}
\ud \bar{X}_t &= A \bar{X}_t\ud t + \sigma_B \ud \bar{B}_t + \bar{\k}_t(\ud Z_t - \frac{C\bar{X}_t + C{\bar{m_t}}}{2}\ud t) \label{eq:Xbart-stochastic} 
\end{align}
where $\bar{\k}_t =\bar{\Sigma}_tC^\top$ is the Kalman gain; the
mean-field terms are the mean $\bar{m}_t=\Expect[\bar{X}_t|\clZ_t]$
and the covariance $\bar{\Sigma}_t
=\Expect[(\bar{X}_t-\bar{m}_t)(\bar{X}_t-\bar{m}_t)^\top|\clZ_t] $; 
$\{\bar{B}_t\}$ is an independent copy of the process noise $\{B_t\}$;
and the initial condition $\bar{X}_0 \sim \mathcal{N}(m_0,\Sigma_0)$.  

\newP{(B) Deterministic linear FPF} The McKean-Vlasov sde is:
\begin{align}
\ud \bar{X}_t &= A m_t\ud t + 
\bar{\k}_t(\ud Z_t - C{\bar{m_t}}\ud t) + \bar{G}_t(\bar{X}_t-\bar{m}_t)\ud t\label{eq:Xbart}
\end{align}
where (as before) $\bar{\k}_t =\bar{\Sigma}_tC^\top$ is the Kalman
gain; the mean $\bar{m}_t=\Expect[\bar{X}_t|\clZ_t]$
and the covariance $\bar{\Sigma}_t
=\Expect[(\bar{X}_t-\bar{m}_t)(\bar{X}_t-\bar{m}_t)^\top|\clZ_t]$; the initial condition $\bar{X}_0 \sim \mathcal{N}(m_0,\Sigma_0)$; and
\begin{equation*}
\bar{G}_t := A- \frac{1}{2}\bar{\k}_t C +
\frac{1}{2}\sigma_B\sigma_B^\top \bar{\Sigma}_t^{-1} + \Om_t\bar{\Sigma}_t^{-1} 
\label{eq:G-gen-N}
\end{equation*} 
where $\Om_t$ is {\em any} skew symmetric $d\times d$ matrix. 



The following Proposition, borrowed from~\cite{AmirACC2016}, shows
that both the filters are exact:


\begin{proposition}
\label{thm_lin} {\em (Theorem~1 in~\cite{AmirACC2016})} Consider the linear Gaussian filtering problem~\eqref{eq:dyn}-\eqref{eq:obs}, and the linear
FPF (Eq.~\eqref{eq:Xbart-stochastic} or Eq.~\eqref{eq:Xbart}).  If $\PP(X_0)=\PP(\bar{X}_0)$ then
\begin{equation*}
\PP(\bar{X}_t|\clZ_t)=\PP(\X_t | \clZ_t),\quad \forall t>0
\end{equation*}
Therefore, $m_t = \bar{m}_t$ and $\Sigma_t = \bar{\Sigma}_t$. 
\label{thm:consistency-FPF-lin}
\end{proposition}

\begin{remark}{({\em Comparison of the deterministic and the stochastic
      FPF})} In the deterministic FPF, there is no explicit Wiener
  process for the process noise. 
  For example, with the choice of $\Om_t=0$, the deterministic linear
  FPF~\eqref{eq:Xbart} has the same terms as the stochastic linear
  FPF~\eqref{eq:Xbart-stochastic}, except that the process noise term
  $\sigma_B \ud \bar{B}_t$ in~\eqref{eq:Xbart-stochastic} is replaced by
  $\frac{1}{2}\sigma_B\sigma_B^\top
  \bar{\Sigma}_t^{-1}(\bar{X}_t-\bar{m}_t)$ in~\eqref{eq:Xbart}.  With
  any Gaussian prior, the term serves to simulate the effect of the process
  noise.  
\end{remark}

\medskip

\begin{remark}[Non-uniqueness] For the vector case ($d>1$), there are
  infinitely many choices of exact control laws, 
  parametrized by the skew-symmetric matrix $\Omega_t$.  In our prior
  work~\cite{AmirACC2016}, the non-uniqueness issues is addressed by introducing
  an optimal transportation cost.  The optimal
  skew-symmetry is shown to be the unique solution of the following
  matrix equation (Proposition~3 in~\cite{AmirACC2016}):
\begin{align*}
\Om_t\bar{\Sigma}_t^{-1} + \bar{\Sigma}_t^{-1}\Om_t & =  (A^\top - A)  \\
& + 
\half(\bar{\Sigma}_tC^\top C - C^\top C\bar{\Sigma}_t)  +  \half (\sigma_B\sigma_B^\top \bar{\Sigma}_t -\bar{\Sigma}_t\sigma_B \sigma_B^\top)
\end{align*}
The choice of $\Omega_t$ does not affect the distribution.  The
optimal skew-symmetry is a correction term that serves
to cancel the skew-symmetry in the dynamics (see Remark~5
in~\cite{AmirACC2016}).  
\end{remark}

\newP{Finite-$N$ FPF algorithm} A particle filter comprises of $N$
stochastic processes (particles) $\{X_t^i:1\le i\le N\}$, where
$X_t^i$ is the state of the
$i^{\text{th}}$-particle at time $t$.  The evolution of $X_t^i$ is
obtained upon empirically approximating the mean-field terms.  In the
following, the finite-$N$ filters for the two cases are described.

\newP{(A) Finite-$N$ stochastic FPF} The evolution of $X_t^i$ is given by
the sde:
\begin{align}
\ud X^i_t &= A X^i_t\ud t + \sigma_B \ud B_t^i + \k^{(N)}_t(\ud Z_t -
\frac{CX^i_t + Cm^{(N)}_t}{2}\ud t) 
\label{eq:Xit-s}
\end{align}
where $\k^{(N)}_t :=\Sigma^{(N)}_tC^\top$; $\{B^i_t\}_{i=1}^N$ are
independent copies of $B_t$; $X^i_0
\stackrel {\text{i.i.d}}{\sim} \mathcal{N}(m_0,\Sigma_0)$ for $i=1,2,\ldots,N$; and the
empirical approximations of the two mean-field terms are as follows:
\begin{align}
m^{(N)}_t&:=\frac{1}{N}\sum_{j=1}^N X^i_t,~ \Sigma^{(N)}_t
:=\frac{1}{N-1}\sum_{j=1}^N (X^i_t-m^{(N)}_t)(X^i_t-m^{(N)}_t)^\top
\label{eq:empr_app_mean_var}
\end{align}

\newP{(B) Finite-$N$ deterministic FPF} The evolution of $X_t^i$ is
given by the sde:
\begin{align}
\ud {X}_t^i &= A \mN_t \ud t + 
\k^{(N)}_t (\ud Z_t - C m^{(N)}_t \ud t) + G_t^{(N)} ({X}_t^i - m^{(N)}_t
    )\ud t
\label{eq:Xit-d}
\end{align}
where (as before) $\k^{(N)}_t :=\Sigma^{(N)}_tC^\top$; $X^i_0
\stackrel {\text{i.i.d}}{\sim} \mathcal{N}(m_0,\Sigma_0)$; empirical
approximations of mean and variance are defined
in~\eqref{eq:empr_app_mean_var}; and 
\begin{align}
{G}_t^{(N)} & = A - \frac{1}{2}\k^{(N)}_t  C +
\frac{1}{2}\sigma_B \sigma_B^\top({\Sigma}_t^{(N)})^{-1} + \Om_t
({\Sigma}_t^{(N)})^{-1}
\label{eq:GtN}
\end{align}

\medskip
The McKean-Vlasov sdes~\eqref{eq:Xbart-stochastic}
and~\eqref{eq:Xbart} are the respective mean-field limits of the
finite-$N$ filters~\eqref{eq:Xit-s} and~\eqref{eq:Xit-d}. 
Our goals in this paper are as follows: 
\begin{romannum}
\item prove the convergence of the
finite-$N$ filter to its mean-field limit; and 
\item obtain bounds on the error
as a function of the number of particles $N$ and the time $t$.  
\end{romannum}
The convergence and error analysis
relies closely on the classical results on stability of the Kalman filter.  These
are summarized next.

\subsection{Stability of the Kalman filter}

The following is assumed throughout the remainder of this paper:

\newP{Assumption (I)} The system $(A,C)$ is detectable and
$(A,\sigma_B)$ is stabilizable.     

\medskip


\begin{theorem}[Lemma 2.2 and Theorem 2.3 in~\cite{ocone1996}]\label{thm:KF-stability}
Consider the Kalman filter~\eqref{eq:KF-mean}-\eqref{eq:KF-variance}
with initial condition $(m_0,\Sigma_0)$.  Then, under Assumption (I):
\begin{romannum}
\item There exists a solution $\Sigma_{\infty}\succ 0$ to the
  algebraic Riccati equation (ARE)
\begin{align}
A \Sigma_\infty +  \Sigma_\infty A^\top + \sigma_B\sigma_B^{\top} -
\Sigma_\infty C^\top C\Sigma_\infty = 0\label{eq:are}
\end{align}
such that $A - \Sigma_\infty C^\top C$ is Hurwitz.  Let 
\begin{equation}
0<\lambda_0 = \text{min} \{ -\text{Real} \; \lambda : \lambda
\; \text{is an eigenvalue of} \; A - \Sigma_\infty C^\top C\}
\label{eq:lambda0_defn}
\end{equation}
\item The error covariance $\Sigma_t \rightarrow \Sigma_{\infty}$
  exponentially fast for {\em any} initial condition
  ${\Sigma}_0$ (not necessarily the prior): 
\[
\lim_{t\rightarrow \infty}  \Fnorm{\Sigma_t -
  \Sigma_{\infty}} \le \text{(const.)} \; e^{-2\lambda_0 t} \rightarrow 0
\]
\item Starting from two initial conditions $(m_0,\Sigma_0)$ and $(\tilde{m_0},\tilde{\Sigma}_0)$, the means converge in
  the following senses:
\begin{align*}
\lim_{t\rightarrow \infty}  {\sf E} [|m_t - \tilde{m}_t|^2] &\; \le \; \text{(const.)} \; e^{-2\lambda_0 t} \rightarrow 0
\\
\lim_{t\rightarrow \infty} |m_t - \tilde{m}_t| e^{\lambda
  t} &\; = \; 0 \quad {\text{a.s.}}
\end{align*}
for all $\lambda\in (0,\lambda_0)$. 
\end{romannum}
\end{theorem}
\medskip
Throughout this paper, the notation $\Sigma_\infty$ is used to denote
the positive definite solution of the ARE~\eqref{eq:are} and $\lambda_0$ is used to
denote the spectral bound as defined in~\eqref{eq:lambda0_defn}. 


\section{Analysis of empirical mean and covariance}
\label{sec:mean_covariance}

\subsection{Evolution equations}
\label{sec:evol}

We consider the finite-$N$ filters -- Eq.~\eqref{eq:Xit-s} for the
stochastic FPF and Eq.~\eqref{eq:Xit-d} for the deterministic FPF.  The
empirical mean and covariance are defined in
Eq.~\eqref{eq:empr_app_mean_var}.  The error is defined as
\[
\xi^i_t := X^i_t - m^{(N)}_t \quad \text{for}\;\;i=1,2,\ldots,N
\]   
The evolution equations for the mean, covariance, and error are as given
next.  The calculations appear in the 
Appendix~\ref{apdx:mean-var-evolution}.

\newP{(A) Finite-$N$ stochastic FPF} The mean and the covariance evolve
according to the sdes:
\begin{subequations}
\begin{align}
\ud \mN_t &= A\mN_t \ud t + \k^{(N)}_t(\ud Z_t - C\mN_t\ud t) + \frac{\sigma_B}{\sqrt{N}}\ud \tilde{B}_t
\label{eq:mean-evolution-s}\\
\ud \SigN_t &= (A\SigN_t + \SigN_tA^\top + \sigma_B\sigma_B^\top -
{\SigN_t}C^\top C\SigN_t)\ud t + \frac{\ud M_{t}}{\sqrt{N}}
\label{eq:var-evolution-s}
\end{align}
\end{subequations}
where $\tilde{B}_t := \frac{1}{\sqrt{N}}\sum_{i=1}^N B^i_t$ is a
standard Wiener process and $\ud M_t =\frac{\sqrt{N}}{N-1}\sum_{i=1}^N (\xi^i_t {\ud B^i_t}^\top \sigma_B^\top +\sigma_B \ud B^i_t {\xi^i_t}^\top)$ is a matrix-valued martingale with $\Expect[\ud M_t \ud M_t^\top] = (\frac{N}{N-1})^2(\SigN_t\sigma_B\sigma_B^\top + \sigma_B\sigma_B^\top \SigN_t+2\trace(\sigma_B\sigma_B^\top)\SigN_t)\ud t$.

The sde for the error is given by
\begin{align*}
\ud \xi^i_t &= (A - \frac{1}{2}\mathsf{K}^{(N)}_tC)\xi^i_t \ud t +
\sigma_B\ud B^i_t - \frac{\sigma_B}{\sqrt{N}}\ud \tilde{B}_t
\end{align*}

\newP{(B) Finite-$N$ deterministic FPF} The evolution equations are as
follows:
\begin{subequations}
\begin{align}
\ud \mN_t &= A\mN_t \ud t + \k^{(N)}_t(\ud Z_t - C\mN_t\ud t)
\label{eq:mean-evolution-d}\\
\frac{\ud \SigN_t}{\ud t} &= A\SigN_t + \SigN_tA^\top
+\sigma_B\sigma_B^\top- {\SigN_t}C^\top
C\SigN_t\label{eq:var-evolution-d}\\
\frac{\ud \xi^i_t}{\ud t} &= G_t^{(N)}\xi^i_t \nonumber
\end{align}
\end{subequations}
where $G_t^{(N)}$ is defined in~\eqref{eq:GtN}.

\medskip

In the remainder of this paper, the focus is on the error analysis of
the deterministic finite-$N$ FPF algorithm.  The analysis is simpler
because the equations for empirical mean and
covariance~\eqref{eq:mean-evolution-d}-\eqref{eq:var-evolution-d} are
identical to the Kalman
filter~\eqref{eq:KF-mean}-\eqref{eq:KF-variance}. The analysis is seen as
the first step towards the analysis of the more complicated stochastic
FPF which also includes an additional $O(N^{-\half})$ stochastic term
(fluctuation) due to the process noise.  


\begin{remark}
Even though the fluctuations scale as $O(N^{-\half})$, the analysis is
challenging as has been noted in literature (see the remark after Theorem
3.1 in~\cite{delmoral2016stability}).  Error analysis of the ensemble
Kalman filter with noise terms appears
in~\cite{delmoral2016stability} under certain additional techical
assumptions.  Analysis of the deterministic FPF closely follows the
stability theory for Kalman filter.  Related analysis appears in the
recent work~\cite{jana2016stability}.      
\end{remark}

\subsection{Error Analysis}
\label{sec:error_analysis}

\newP{Assumption (II)} 
The initial covariance $\Sigma_0^{(N)}$ is invertible.     

\medskip

The main result for the finite-$N$ deterministic FPF is as follows with
the proof given in Appendix~\ref{apdx:mean-var-error}.

\medskip

\begin{proposition} \label{prop:conv_error}
Consider the Kalman filter~\eqref{eq:KF-mean}-\eqref{eq:KF-variance}
initialized with the prior ${\cal N}(m_0,\Sigma_0)$ and the
finite-$N$ deterministic FPF~\eqref{eq:Xit-d} initialized with
$X_0^i\stackrel{\text{i.i.d}}{\sim} {\cal N}(m_0,\Sigma_0)$ for $i=1,2,\ldots,N$.  Under
Assumption (I)-(II), the following characterizes the convergence and
error properties of the empirical mean and covariance
$(m_t^{(N)},\Sigma_t^{(N)})$ obtained from the finite-$N$ filter to the mean and
covariance $(m_t,\Sigma_t)$ obtained from the Kalman filter:
\begin{romannum}
\item{Convergence:} For any finite $N$, as $t\rightarrow\infty$:
\begin{align*}
\lim_{t \to \infty} e^{\lambda t}|\mN_t - m_t| &= 0\quad \text{a.s}\\
\lim_{t \to \infty} e^{2\lambda t}\Fnorm{\SigN_t - \Sigma_t} &= 0\quad \text{a.s}
\end{align*} 
for all $\lambda\in (0,\lambda_0)$. 
\item{Mean-squared error:} For any $t>0$, as $N\rightarrow\infty$:
\begin{subequations}
\begin{align}
\Expect[|\mN_t-m_t|^2]&\leq \frac{(c_1\Expect[|X_0-m_0|^2]+c_2\Expect[|X_0-m_0|^4])e^{-2\lambda_0 t}}{N}  \label{eq:mean-estimate}\\
\Expect[\Fnorm{\SigN_t-\Sigma_t}^2]&\leq \frac{c_3\Expect[|X_0-m_0|^4]e^{-4 \lambda_0 t}}{N}   \label{eq:var-estimate}
\end{align} 
\end{subequations}
where $c_1,c_2,c_3$ are positive constants.  For the scalar ($d=1$)
case, one has the following explicit formulae for these constants:
\begin{align*}
c_1=e^{|\log(\beta)|},~~ c_2=\frac{C^2}{2\lambda_0}\beta^2e^{|\log(\beta)|}(1-e^{-2\lambda_0t}),~~ c_3=\beta^2
\end{align*}
where $\lambda_0=(A^2+\sigma_B^2C^2)^{\half}$ and 
$\beta =(\frac{2\lambda_0}{\lambda_0-A})^2$. 
\end{romannum} 


\end{proposition}

\medskip

\begin{remark}
Asymptotically (as $t \to \infty$) the empirical mean and variance of
the finite-$N$ filter becomes exact.  This is because of the stability
of the Kalman filter whereby the filter forgets the initial
condition. In fact, the i.i.d assumption on the initial condition
$X_0^i$ is not necessary to obtain this conclusion.  
\end{remark}
\begin{remark}
The assumption (II) on the invertibility of
the initial covariance $\SigN_0$ can be relaxed to $\text{Ker}(\SigN_0)\subseteq\text{Ker}(\sigma_B\sigma_B^\top)$ (in the
proof, one works with the pseudo-inverse instead of the inverse).  The
latter is important because ensemble Kalman filters are often
simulated in high-dimensional settings where the number of particles $N$
may be smaller than the dimension $d$ of the problem~\cite{reich11}.    
\end{remark}



    



\subsection{Numerics}

Consider a scalar ($d=1$) linear filtering problem~\eqref{eq:dyn}-\eqref{eq:obs}
with parameters $A=0.1$, $\sigma_B=1.0$, $C=1.0$, $m_0=3.0$, and $\Sigma_0=5.0$.  
The stochastic linear FPF~\eqref{eq:Xit-s} and the deterministic
linear FPF~\eqref{eq:Xit-d} are simulated for this problem.  The
ground truth $(m_t,\Sigma_t)$ is obtained by simulating a Kalman
filter~\eqref{eq:KF-mean}-\eqref{eq:KF-variance}.   

Figure~\ref{fig:traj-FPF} depicts the results for a single simulation
of the deterministic FPF algorithm with $N=100$ particles. The
trajectory $X_t^i$ of the $N$ particles along with the empirical mean $\mN_t$, the
empirical variance $\SigN_t$, the conditional mean $m_t$, and the
conditional variance $\Sigma_t$ are depicted in the figure.
Consistent with the conclusion of the Prop.~\ref{prop:conv_error},
both the empirical mean and variance are seen to converge
exponentially fast.
    
Figure~\ref{fig:conv-t} and Figure~\ref{fig:conv-N} provide a simulation-based
illustration of the mean-squared error as a function of $t$ and $N$,
respectively.  A description of these plots is as follows:
\begin{romannum}
\item 
An empirical estimate of the error
$\Expect[|\mN_t-m_t|^2]$ as a function of time is depicted in part~(b)
of the figure.  The empirical estimates are obtained by running
$M=1000$ simulations of the FPF with $N=100$ particles.  The
upper bound is also depicted.  Consistent with the result in
Prop.~\ref{prop:conv_error},  the error converges exponentially fast
to zero for the deterministic FPF.  For the stochastic FPF, the error
reaches a steady state value because of the presence of the process
noise term.    
\item In the part~(c) of the figure, empirical estimates of the error
  are depicted now as a function of $N$ for a fixed $t=2.0$.  As
  before, these are estimated with $M=1000$ simulations.  For both the
  deterministic and the stochastic cases, the error is seen to scale as
  $O(\frac{1}{N})$ consistent with the scaling given in
  Prop.~\ref{prop:conv_error}. 
\end{romannum}





\begin{figure*}[t]
\begin{tabular}{ccc}
\subfigure[]{
\includegraphics[width=0.66\columnwidth]{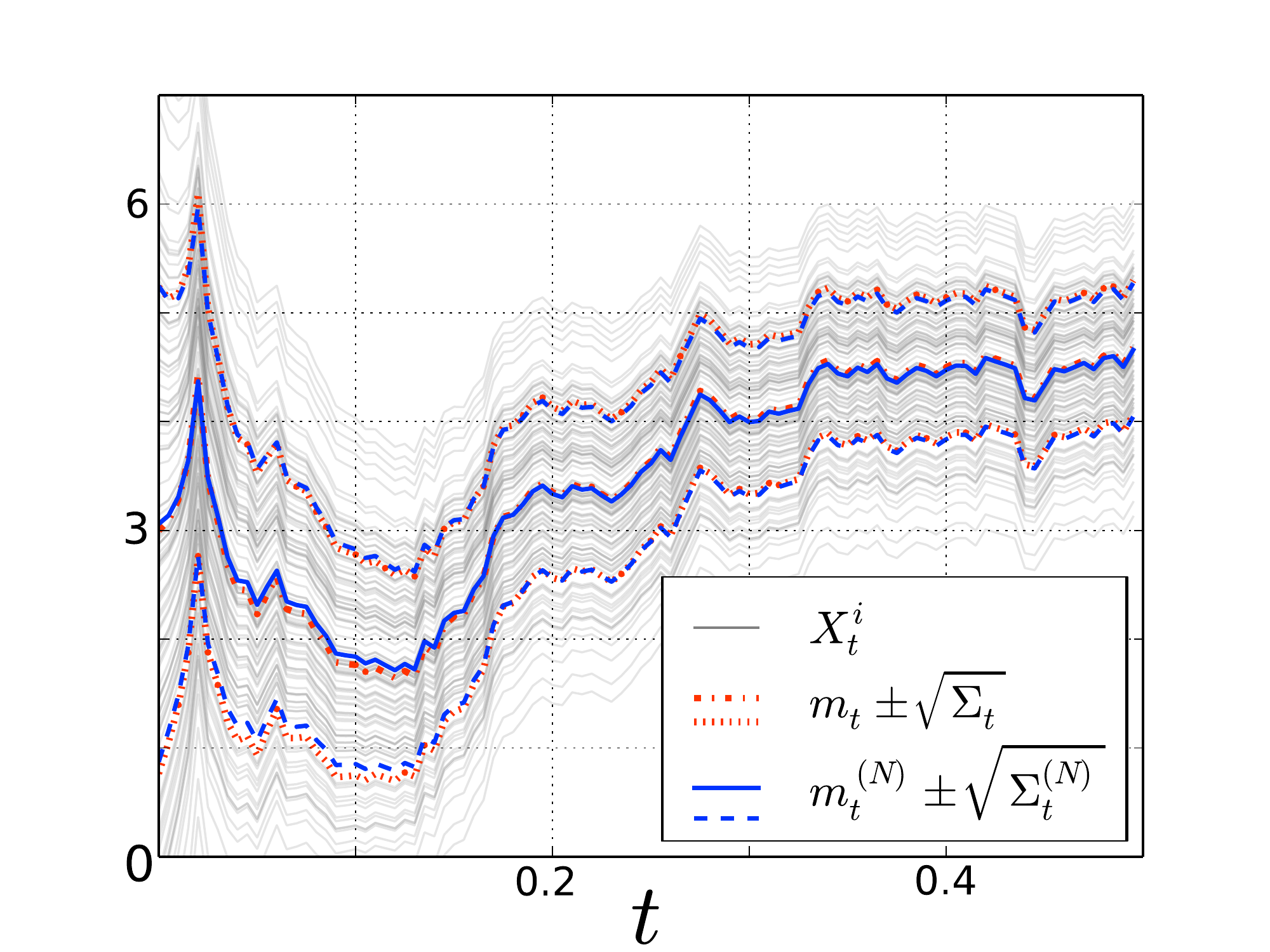}
\label{fig:traj-FPF}
}&
\subfigure[]{
\includegraphics[width=0.66\columnwidth]{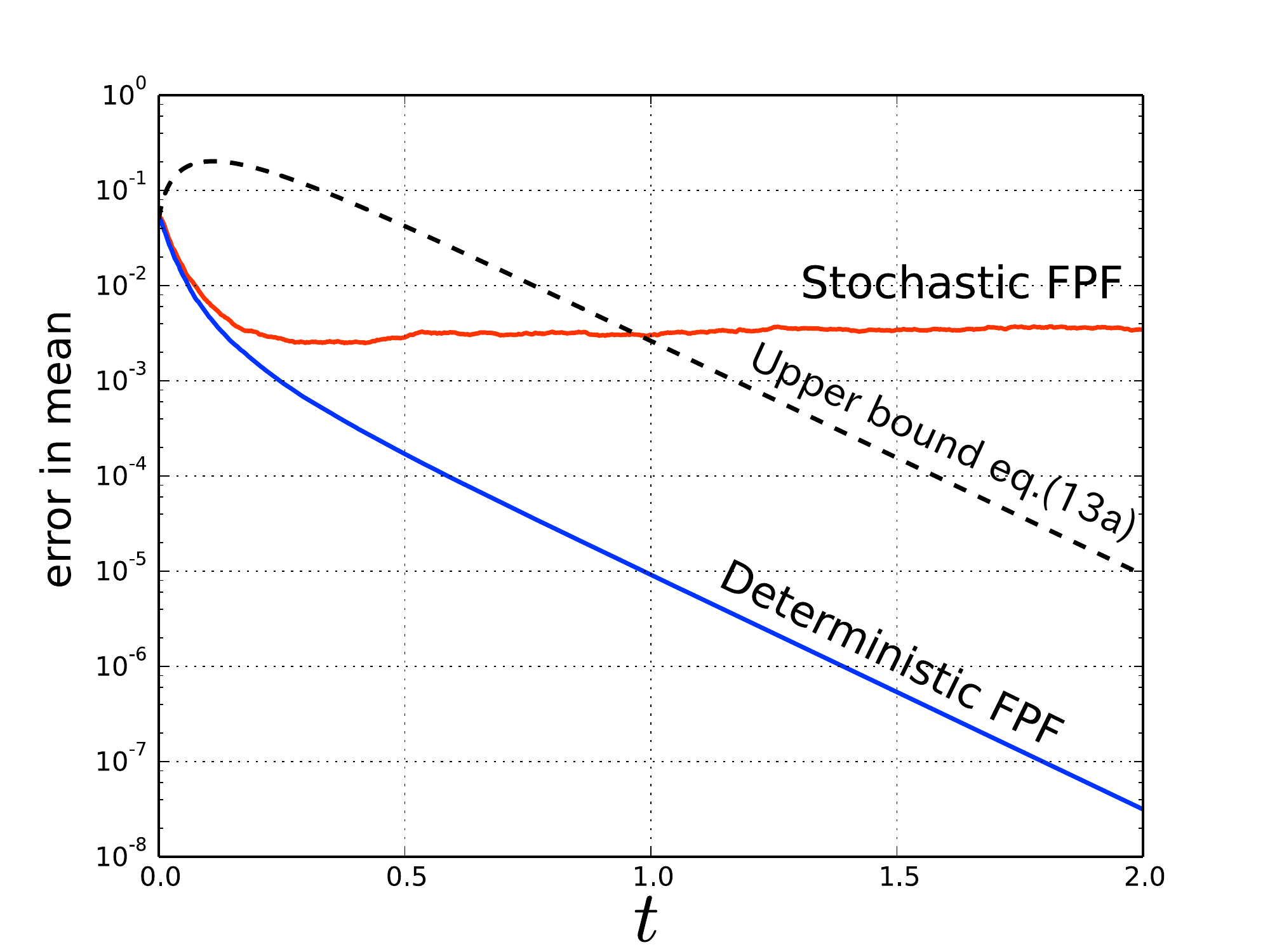}
\label{fig:conv-t}
}&
\subfigure[]{
\includegraphics[width=0.66\columnwidth]{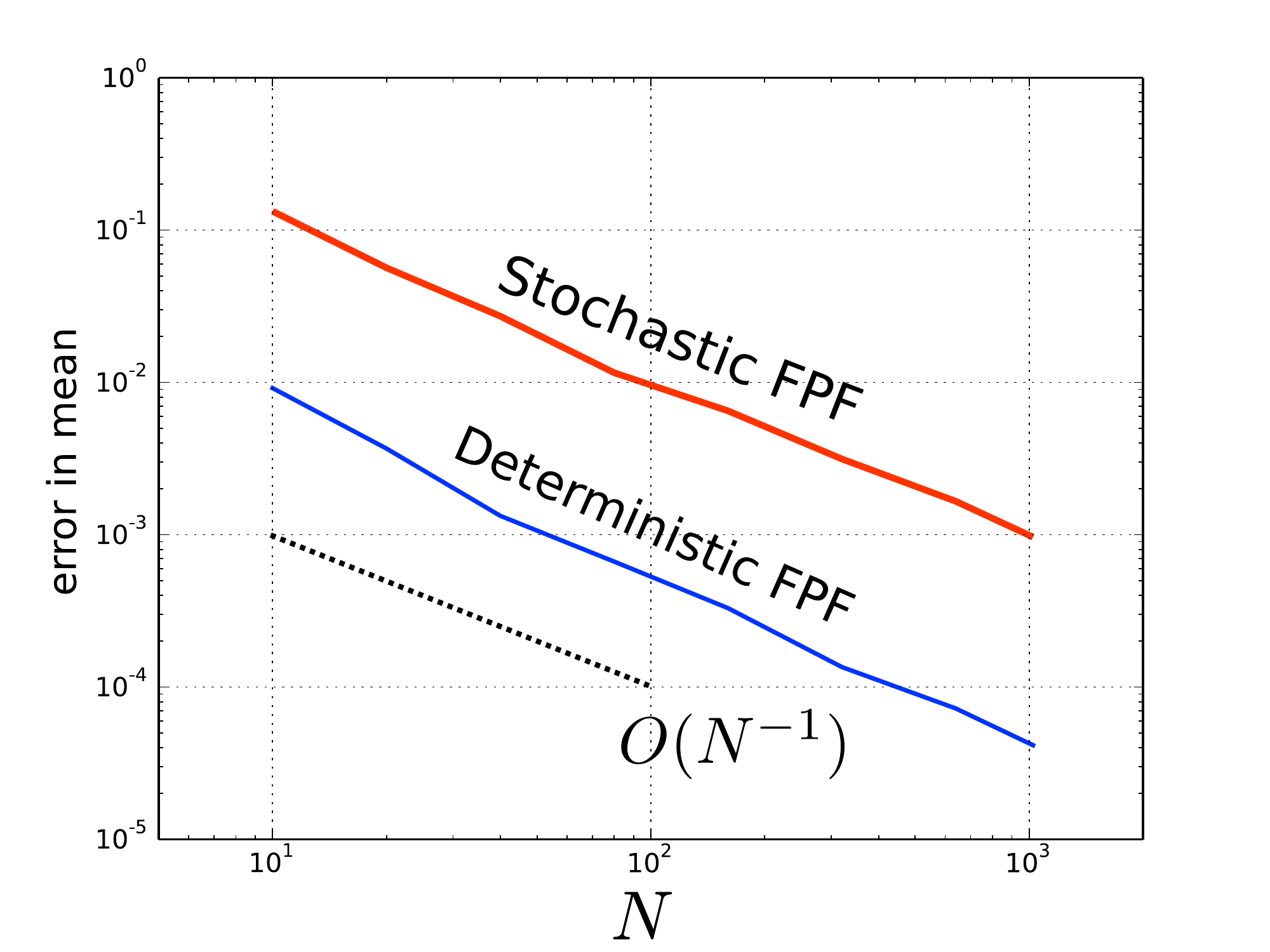}
\label{fig:conv-N}
}
\end{tabular}
\caption{Results of numerical experiment: (a) Trajectories of $N=100$
  particles together with the empirical mean, the empirical variance,
  the (true) conditional mean, and the conditional variance for a
  single simulation; (b) Monte-Carlo estimate of the mean-squared
  error (in estimating the mean) as a function of time $t$; (c) Monte-Carlo estimate of the mean-squared
  error as a function of $N$.  The Monte-Carlo estimates are obtained
  by running $M=1000$ simulations.}
\end{figure*}
\section{Propagation of chaos}
\label{sec:poa}

At the initial time $t=0$, the particles $\{X_0^i\}_{i=1}^N$ are
sampled i.i.d. from the prior distribution.  In any finite-$N$
implementation of the filter, the i.i.d. property is
destroyed for $t>0$ because of the interactions:  For the linear
FPFs~\eqref{eq:Xit-s} and~\eqref{eq:Xit-d}, the interaction terms
are a function of the empirical mean $\mN_t$ and the empirical covariance
$\SigN_t$.  Since these terms depend upon all the particles, the
$i^{\text{th}}$ particle in the population is coupled to/interacts
with (the randomness of) all other particles.  
Even though the particles are no longer i.i.d for any finite choice of $N$, one (formally) expects the particles to become approximately
i.i.d (in a sense that needs to be made precise) for large $N$.
Intuitively, this is because as $N \to \infty$, $\mN_t \to m_t$ and
$\SigN_t \to \Sigma_t$.  And for the limiting mean-field model, the
particles are i.i.d for $t>0$ provided they are i.i.d. at the initial
time $t=0$.  The phenomenon is referred to as the {\em
  propagation of chaos} whereby the chaos (i.i.d property of the
population) propagates through time.

The mathematical definitions are as follows:  Denote  $E:=\Re^d\times[0,\infty)$.  Let $\mu_N$ be
the probability measure on $E^N$ associated with the process
$(X^1,\ldots,X^N)$. Let $\bar{\mu}$ be the probability measure on $E$
associated with the mean-field solution $\bar{X}$.  Then $\mu_N$ is
said to be {\em $\bar{\mu}$-chaotic} if 
\begin{equation*}
\pi_k\mu_N \overset{\text{weak}}{\longrightarrow} \bar{\mu}^{(k)}\quad \text{as}\quad N\to \infty
\end{equation*} 
where $\pi_k \mu_N$ is the $k$-marginal distribution,
$\bar{\mu}^{(k)}$ is the $k$-fold product, and the convergence is in
the weak sense.  A somewhat easier formulation of this condition
appears in~\cite[Proposition 2.2]{sznitman1991} as
\begin{equation}
\lim_{N \to \infty} \Expect\left[\left|\frac{1}{N}\sum_{i=1}^N f(X^i)  - \Expect[f(\bar{X})] \right|^2\right]=  0
\label{eq:LLN}
\end{equation} 
for all bounded functionals $f:E\to\Re$.  

\begin{remark} 
Some difficulties in carrying out the propagation of chaos analysis
for the FPF are as follows:  (i) The drift term in the evolution equation for the covariance
is not Lipschitz; For the stochastic FPF~\eqref{eq:Xit-s}, the
noise terms (the martingale $M_t$) depend upon the state.  In our
analysis, we circumvent some of these difficulties by limiting to the
scalar ($d=1$) case where explicit solution of the covariance is
available.  As was the case in Sec.~\ref{sec:error_analysis}, we focus on the deterministic FPF where the terms due to
the process noise are not present.  Even in this special case, we show
the convergence for the marginal distribution {\em only} for fixed
time $t>0$.   
That is, we show 
\begin{equation}
\lim_{N \to \infty} \Expect\left[\left|\frac{1}{N}\sum_{i=1}^N
    f(X^i_t)  - \Expect[f(\bar{X}_t)] \right|^2\right]=  0
\label{eq:LLN-t}
\end{equation}
for all bounded functions $f:\Re^d \to \Re$.  Extension to estimates
that are uniform in time in the general settings is a subject of continuing work.    
\end{remark}

Derivation of error estimates involve construction of $N$ independent
copies of the mean-field equation~\eqref{eq:Xbart} corresponding to the
deterministic FPF~\eqref{eq:Xit-d}.  Consistent with
our convention to denote mean-field variables with a bar, the
stochastic processes are denoted as $\{\bar{X}_t^i:1\le i \le N\}$
where $\bar{X}^i_t$ denotes the state of the $i^{\text{th}}$ particle
at time $t$.  The particle evolves according to the mean-field
equation~\eqref{eq:Xbart} as
\begin{align}
\ud \bar{X}^i_t & = A \bar{m}_t \ud t + 
\bar{\k}_t(\ud Z_t - C{\bar{m}_t}\ud t) + \bar{G}_t(\bar{X}_t^i -\bar{m}_t)\
 \label{eq:barXit}
\end{align} 
where $\bar{\k}_t =\bar{\Sigma}_tC^\top$ is the Kalman
gain and the initial condition $\bar{X}^i_0=X^i_0$ -- the initial condition of
the $i^{\text{th}}$ particle in the finite-$N$ FPF~\eqref{eq:Xit-d}. 
%
The mean-field process $\bar{X}^i_t$ is thus coupled to $X^i_t$
through the initial condition.
The following Proposition characterizes the error between $X^i_t$ and
$\bar{X}^i_t$ (the estimate is essential for the propagation of
chaos analysis). The proof appears in the
Appendix~\ref{apdx:prop-chaos}.

%

\begin{proposition} \label{prop:prop-chaos}
Consider the stochastic processes $X^i_t$ and $\bar{X}^i_t$ whose
evolution is defined 
according to the deterministic FPF~\eqref{eq:Xit-d} and its mean-field
model~\eqref{eq:barXit}, respectively. The initial condition $X_0^i\stackrel{\text{i.i.d}}{\sim} {\cal N}(m_0,\Sigma_0)$ for $i=1,2,\ldots,N$ and the dimension $d=1$.  Then under
Assumptions~(I)-(II):  
\begin{romannum}
\item The explicit solution is given as
\begin{subequations}
\begin{align*}
X^i_t  &= \mN_t + \left(\frac{\SigN_t}{\SigN_0}\right)^{\half}(X^i_0-\mN_0)\\
\bar{X}^i_t  &= m_t + \left(\frac{\Sigma_t}{\Sigma_0}\right)^\half(X^i_0-m_0)
\end{align*}
\end{subequations}
\item For a fixed $t>0$, in the limit as $N \to \infty$
\begin{equation}
\Expect[|X^i_t-\bar{X}^i_t|^2] \leq
\frac{\text{(const.)}}{{N}}
\label{eq:estimate_POA}
\end{equation}  
\end{romannum}
\end{proposition}
\medskip

The estimate~\eqref{eq:estimate_POA} is used to prove the following
important result that the empirical distribution of the particles in
the linear FPF converges weakly to the true posterior distribution.
Its proof appears in the Appendix~\ref{apdx:prop-chaos}.

\begin{corollary} \label{cor:prop-chaos}
Consider the linear filtering problem~\eqref{eq:dyn}-\eqref{eq:obs}
and the finite-$N$ deterministic FPF~\eqref{eq:Xit-d}. The initial condition $X_0^i\stackrel{\text{i.i.d}}{\sim} {\cal N}(m_0,\Sigma_0)$ for $i=1,2,\ldots,N$ and the dimension $d=1$. Under Assumptions (I) and (II), for any Lipschitz
function $f:\Re^d \to \Re$, in the asymptotic limit as $N\to \infty$
\begin{equation*}
\Expect\left[\left|\frac{1}{N}\sum_{i=1}^N
    f(X^i_t)-\Expect[f(X_t)|\clZ_t]\right|^2\right] \leq \frac{\text{(const.)}}{{N}} 
\end{equation*}
\end{corollary}

\bibliographystyle{plain}
\bibliography{fpfbib,ref,fpfbib2,Optimization,meanfield,meanfield_v2} 
 
\appendix
\subsection{Derivation of evolution equations in \Sec{sec:evol}}\label{apdx:mean-var-evolution}

\newP{(A) Finite-$N$ stochastic FPF} Consider Eq.~\eqref{eq:Xit-s} for
the $i^{\text{th}}$ particle.  Summing up over the index
$i=1,\ldots,N$ and dividing by $N$, Eq.~\eqref{eq:mean-evolution-s} for the mean is obtained.
To obtain~\eqref{eq:var-evolution-s}, first define $\xi^i_t := X^i_t -
\mN_t$. Therefore,  
\begin{equation*}
\ud \xi^i_t = (A - \frac{1}{2}\mathsf{K}^{(N)}_tC)\xi^i_t \ud t + \sigma_B\ud B^i_t - \frac{1}{N}\sum_{j=1}^N \sigma_B\ud B^j_t 
\end{equation*}
and
\begin{align*}
\ud (\xi^i_t {\xi^i_t}^\top) = &(A - \frac{1}{2}\mathsf{K}^{(N)}_tC)\xi^i_t{\xi^i_t}^{\top} \ud t+ \xi^i_t {\xi^i_t}^\top (A - \frac{1}{2}\mathsf{K}^{(N)}_tC)^\top \ud t \\&+\frac{N-1}{N} \sigma_B\sigma_B^\top \ud t + \xi^i_t (\ud B_t^i-\frac{1}{N}\sum_{j=1}^N\ud B^j_t)^\top \sigma_B^\top\\
&+\sigma_B (\ud B_t^i-\frac{1}{N}\sum_{j=1}^N\ud B^j_t){\xi^i_t}^\top
\end{align*}
Summing over $i=1,\ldots,N$ and dividing by $(N-1)$ gives
\begin{align*}
\ud \SigN_t= &(A-\frac{1}{2}\k_t^{(N)}C) \SigN_t \ud t + \SigN_t(A-\frac{1}{2}\k_t^{(N)}C)^\top \ud t \\
&+ \sigma_B\sigma_B^\top\ud t + \frac{1}{N-1}\sum_{i=1}^N \xi^i_t {\ud B^i_t}^\top \sigma_B^\top + \frac{1}{N-1}\sum_{i=1}^N \sigma_B \ud B^i_t {\xi^i_t}^\top 
\end{align*}
which is Eq.~\eqref{eq:var-evolution-s} for the covariance.

\newP{(B) Finite-$N$ deterministic FPF} Eq.~\eqref{eq:Xit-d} is
obtained as before by summing up Eq.~\eqref{eq:Xit-d} for the
$i^{\text{th}}$ particle from $i=1,\ldots,N$.
The equation for the empirical mean is simply obtained by summing up
the equations~\eqref{eq:Xit-d} for $i=1,\ldots,N$.  To obtain the
equation for the empirical covariance, first define $\xi^i_t := X^i_t - \mN_t$.
Therefore,  
$\ud \xi^i_t = G_t^{(N)}\xi^i_t \ud t$
and
\begin{align*}
\ud (\xi^t{\xi^i_t}^\top) = G_t^{(N)}\xi^i_t{\xi^i_t}^\top\ud t + \xi^i_t{\xi^i_t}^\top {G_t^{(N)}}^\top \ud t
\end{align*}
Summing over $i=1,\ldots,N$ and dividing by $(N-1)$ gives
\begin{align*}
\frac{\ud \SigN_t}{\ud t} &= G_t^{N}\SigN_t + \SigN_t {G_t^{(N)}}^\top 
\end{align*}
which is Eq.~\eqref{eq:var-evolution-d} for the covariance.  

It is noted that $G_t^{(N)}$ is well-defined because $\SigN_0$ and
thus $\SigN_t$ is invertible because of Assumption (II).  

\subsection{Proof of the Prop.~\ref{prop:conv_error}}\label{apdx:mean-var-error}

Since the equations for the empirical mean~\eqref{eq:mean-evolution-d}
and the empirical covariance~\eqref{eq:var-evolution-d} are identical to
the Kalman filter~\eqref{eq:KF-mean}-\eqref{eq:KF-variance}, the
a.s. convergence of mean and variance follows from the filter
stability theory (see Theorem~\ref{thm:KF-stability}).  In the
following, mean-squared estimates are derived for the large-$N$ limit.  
We begin with the scalar ($d=1$) case:

\medskip

\newP{Scalar case} The explicit solution of~\eqref{eq:var-evolution-d} is given by
\begin{equation}
\SigN_t  = \Sigma_\infty + \frac{e^{-2\lambda_0 t}}{\frac{1}{\SigN_0-\Sigma_\infty}+\frac{C^2}{2\lambda_0}(1-e^{-2\lambda_0t})} =: f(\SigN_0,t)\label{eq:SigNt}
\end{equation}
where $\lambda_0 =(A^2+\sigma_B^2C^2)^\half$ and $\Sigma_\infty = \frac{A+\lambda_0}{C^2}$.  The function $f(x,t)$ is Lipschitz with respect to $x$ with Lischitz constant $\beta e^{-2\lambda_0t}$ where
 $\beta=(\frac{2\lambda_0}{\lambda_0-A})^2$. Therefore 
\begin{equation*}
|\SigN_t-\Sigma_t| \leq \beta e^{-2\lambda_0 t}|\SigN_0-\Sigma_0|
\end{equation*}
Hence 
\begin{align*}
\Expect[|\SigN_t-\Sigma_t|^2]&\leq \beta^2e^{-4\lambda_0 t}\Expect[|\SigN_0-\Sigma_0|^2]\\&\leq \beta^2e^{-4\lambda_0 t}\frac{\Expect[|X_0-m_0|^4]-\Sigma^2_0}{N} + O(\frac{1}{N^2})
\end{align*}
which gives the result~\eqref{eq:var-estimate} for the scalar case.

\medskip

The estimate~\eqref{eq:mean-estimate} for the mean is obtained next.  Define $\delta m_t :=
\mN_t-m_t$ and $\delta \Sigma_t :=
\SigN_t-\Sigma_t$. Using-\eqref{eq:mean-evolution-d} and~\eqref{eq:KF-mean},
\begin{align*}
\ud \delta m_t = &(A-C^2\SigN_t)\delta m_t \ud t +\delta \Sigma_t C\ud I_t
\end{align*}
where $\ud I_t := \ud Z_t  - Cm_t \ud t$ is the innovation process. Therefore, 
\begin{align*}
\delta m_t = e^{\int_0^t (A-C^2\SigN_s)\ud s} \delta m_0 + \int_0^t e^{\int_s^t (A-C^2\SigN_\tau)\ud \tau } C\delta \Sigma_s \ud I_s
\end{align*}
Squaring the expression and taking the expectation yields
\begin{align*}
\Expect[|\delta m_t|^2] &= e^{2\int_0^t (A-C^2\Sigma_s)\ud s} \Expect[|\delta m_0|^2] \\&+ \int_0^t e^{2\int_s^t (A-C^2\Sigma_\tau)\ud \tau} C^2\Expect[|\delta \Sigma_s|^2] \ud t
\end{align*}
where we have used the fact that the innovation process $I_t$ is a
Wiener process~\cite[Lemma 5.6]{xiong2008}. 
Then using the bound 
\begin{align*}
\int_s^t&(A-C^2\SigN_s)\ud s\\&= -\lambda_0(t-s) + \log(\frac{1+(\SigN_0-\Sigma_\infty)\frac{C^2}{2\lambda_0}(1-e^{-2\lambda_0s}) }{1+(\SigN_0-\Sigma_\infty)\frac{C^2}{2\lambda_0}(1-e^{-2\lambda_0t}) })
\\
&\leq -\lambda_0(t-s)+\frac{1}{2}|\log(\beta)|
\end{align*}
and $ \Expect[|\delta \Sigma_s|^2] \leq \beta^2e^{-4\lambda_0 t} \Expect[|\delta \Sigma_0|^2]$
we obtain
\begin{align*}
\Expect[|\delta m_t|^2] \leq &e^{-2\lambda_0 t+|\log(\beta)|}\Expect[|\delta m_0|^2]\\&+C^2\beta^2e^{-2\lambda_0 t+|\log(\beta)|}\frac{1}{2\lambda_0}(1-e^{-2\lambda_0 t})\Expect[|\delta \Sigma_0|^2]
\end{align*}
The mean-squared estimate~\eqref{eq:mean-estimate} for the mean in the scalar case follows from noting $\Expect[|\delta m_0|^2] = \frac{\Expect[|X_0-m_0|^2]}{N}$ and $\Expect[|\delta \Sigma_0|^2]\leq \frac{\Expect[|X_0-m_0|^4}{N}$.

\medskip

\newP{Vector case} The explicit solution of the Riccati equation in
the vector case is given by~\cite[pp. 149]{brockett2015finite}
\begin{align}
\SigN_t = \Sigma_\infty + e^{F_\infty t}D_t^{-1}e^{F_\infty^\top t}=:f(\SigN_0,t)\label{eq:SigNt-vec}
\end{align}
where $F_\infty := A - \Sigma_\infty C^\top C$, and 
\begin{align*}
D_t &:=(\SigN_0-\Sigma_\infty)^{-1}+\int_0^t e^{F_\infty^\top s}C^\top Ce^{F_\infty s}\ud s
\end{align*}
The function $f(x,t)$ is Lipschitz with respect to to $x$ with Lipschitz constant $\beta_de^{-2\lambda_0t}$ where
$\beta_d := \sup_{\SigN_0} 
\snorm{(D_t(\SigN_0-\Sigma_\infty))^{-1}}^2$
Therefore,
\begin{align*}
\Fnorm{\SigN_0-\Sigma_t} \leq& \beta_de^{-2\lambda_0t}\Fnorm{\SigN_0-\Sigma_0}
\end{align*}
Squaring both sides and taking expectation gives the covariance error estimate:
\begin{align*}
\Expect[\Fnorm{\SigN_t-\Sigma_t}^2] \leq \beta_d^2e^{-4\lambda_0 t} \frac{\Expect[|X_0-m_0|^4]-\trace(\Sigma_0^2)}{N}+ O(\frac{1}{N^2})
\end{align*}
where we have used 
$\Expect[\Fnorm{\SigN_0-\Sigma_0}^2]=\frac{\Expect[|X_0-m_0|^4]-\trace(\Sigma_0^2)}{N} + O(N^{-2})$.


\medskip

The procedure for obtaining the error estimate for the mean is also as before.
As in the scalar case,
\begin{align*}
\delta m_t = \Phi_{t,0} \delta m_0 + \int_0^t \Phi_{t,s}\delta\Sigma_s C^\top\ud I_s
\end{align*}
where $\Phi_{t,s}$ is the state transition matrix for $F^{(N)}_t=A-\SigN_tC^\top C$. The expected norm-squared of $\delta m_t$ is 
\begin{align}
\Expect[|\delta m_t|^2] = & \Expect[\delta m_0^\top \Phi_{t,0}^\top
\Phi_{t,0}\delta m_0]  \nonumber \\&+ \int_0^t \trace(C\delta \Sigma_s
\Phi_{t,s}^\top\Phi_{t,s}\delta \Sigma_s C^\top)\ud s
\label{eq:mean_sq_11}
\end{align}
where we used the fact that the innovation process $I_t$ is a Wiener process. Expressing 
$F_t^{(N)} = F_\infty - e^{F_\infty t}D_t^{-1}e^{F^\top_\infty
  t}C^\top C$, 
its spectral norm is bounded as
$\snorm{F_t^{(N)}} \leq -\lambda_0+ e^{-2\lambda_0  t}2\lambda_0 c_0$
where $c_0:=\frac{1}{2\lambda_0}\snorm{C^\top C}\sup_{\SigN_0}\snorm{D_t^{-1}}$.
Therefore $\snorm{\Phi_{t,s}}\leq e^{-\lambda_0 t + c_0}$.
Use this inequality in~\eqref{eq:mean_sq_11} to conclude
\begin{align*}
\Expect[|\delta m_t|^2]& = e^{-2\lambda_0 t+2c_0}\Expect[|\delta m_0|^2] + \\& +e^{-2\lambda_0 t+2c_0}\frac{1-e^{-2\lambda_0 t}}{2\lambda_0}\snorm{CC^\top}\beta_d^2\Expect[\Fnorm{\delta \Sigma_0^2}^2]
\end{align*}
The error bound~\eqref{eq:mean-estimate} follows from noting, as also in the scalar case,
$\Expect[|\delta m_0|^2]= \frac{1}{N}\Expect[|X_0-m_0|^2]$ and $\Expect[\|\delta \Sigma_0\|_F^2]\leq \frac{\Expect[|X_0-m_0|^4}{N}$.
 

\subsection{Proofs of the Prop.~\ref{prop:prop-chaos} and Cor.~\ref{cor:prop-chaos}}\label{apdx:prop-chaos}
\begin{proof}
Part (i): Use the decomposition 
\begin{equation*}
X^i_t = \mN_t + \xi^i_t,\quad \bar{X}^i_t = m_t + \bar{\xi}^i_t
\end{equation*}
Recall that 
\begin{align*}
\ud {\xi}^i_t &= G_t^{(N)}{\xi}^i_t \ud t,\quad 
\ud \bar{\xi}^i_t = G_t\bar{\xi}^i_t\ud t
\end{align*} 
Hence, for the scalar case
\begin{align*}
{\xi}^i_t &= e^{\int_0^tG_s^{(N)}\ud s}{\xi}^i_0,\quad 
\bar{\xi}^i_t = e^{\int_0^tG_s\ud s}\bar{\xi}^i_0
\end{align*} 
By definition, $G^{(N)}_t$ and $G_t$ satisfy
\begin{align*}
2\SigN_t G_t = \frac{\ud \SigN_t}{\ud t},\quad 2\Sigma_t G_t = \frac{\ud \Sigma_t}{\ud t}
\end{align*}
Therefore 
\begin{align*}
\int_0^tG^{(N)}_s \ud s = \frac{1}{2}\log(\frac{\SigN_t}{\SigN_0}),\quad \int_0^tG_s \ud s = \frac{1}{2}\log(\frac{\Sigma_t}{\Sigma_0})
\end{align*}
which concludes the result for part (i) of the Proposition.

Part (ii): Use the triangle inequality to conclude
\begin{equation*}
\Expect[|X^i_t-\bar{X}^i_t|^2]^{\frac{1}{2}} \leq \Expect[|\xi^i_t-\bar{\xi}^i_t|^2]^{\frac{1}{2}} +\Expect[|\mN_t-m_t|^2]^{\frac{1}{2}} 
\end{equation*}
The error between $\mN_t$ and $m_t$ is already obtained as part of the Prop.~\ref{prop:conv_error} given by~\eqref{eq:mean-estimate}. 
Using the explicit solutions in part (i), the other term is the $L^2$-norm of  
\begin{align*}
\xi^i_t - \bar{\xi}^i_t &=\underbrace{(\left(\frac{\SigN_t}{\SigN_0}\right)^{\frac{1}{2}}-\left(\frac{\Sigma_t}{\Sigma_0}\right)^{\frac{1}{2}}){\xi}^i_0}_{\text{(I)}}+ \underbrace{\left(\frac{\Sigma_t}{\Sigma_0}\right)^{\frac{1}{2}}(\xi^i_0-\bar{\xi}_0)}_{\text{(II)}}
\end{align*}
A bound on the second term~(II) is obtained as follows:
%
\begin{align*}
\Expect\left[\frac{\Sigma_t}{\Sigma_0}|\xi^i_0-\bar{\xi}_0|^2\right]\leq
\frac{\Sigma_t}{\Sigma_0}\Expect\left[|\xi^i_0-\bar{\xi}_0|^2\right]
&=\frac{\Sigma_t}{\Sigma_0}\Expect\left[|\mN_0-m_0|^2\right]\\&= \frac{\Sigma_t}{\Sigma_0} \frac{\Expect[|X_0-m_0|^2]}{N}
\end{align*}
where we used the identity $\xi^i_0+\mN_0=\bar{\xi}^i_0+m_0$ because $X^i_t=\bar{X}^i_t$.
The bound the first term (I) is involved. Define $g(x,t):=\sqrt{\frac{f(x,t)}{x}}$ where $f$ is defined in~\eqref{eq:SigNt}. We are interested in obtaining a bound on
$\Expect[|g(\SigN_0,t)-g(\Sigma_0,t)|^2{\xi^i_0}^2]$. Denote
$r(\omega):=g(\SigN_0(\omega),t)-g(\Sigma_0,t)$ and express
\begin{equation*}
r= r\mathds{1}_S + r\mathds{1}_{S^c}
\end{equation*}
where the event $S:=\{\omega\;:\;\SigN_0(\omega)\geq
\frac{1}{2}\Sigma_0\}$. On $S$, 
the function $g(x,t)$ is Lipschitz with respect to $x$ with Lipschitz constant $L_g=\frac{\beta e^{-2\lambda_0 t}}{2\min(\Sigma_0,\Sigma_\infty)}+\frac{2\sqrt{\Sigma_\infty}}{\min(\Sigma_\infty,\Sigma_0)^{3/2}}$. On the complement $S^c$:
\begin{align*}
\Expect[r^2{\xi^i_0}^2\mathds{1}_{S^c}]&\leq 
 2\Expect\left[\frac{\SigN_t}{\SigN_0}{\xi^i_0}^2\mathds{1}_{S^c}\right] + 2\Expect\left[\frac{\Sigma_t}{\Sigma_0}{\xi^i_0}^2\mathds{1}_{S^c}\right]
\\&\leq 2N\Expect[\SigN_t\mathds{1}_{S^c}] + 2N\frac{\Sigma_t}{\Sigma_0}
\Expect[\SigN_0\mathds{1}_{S^c}]\\&\leq 4N(\Sigma_\infty + \Sigma_0)\P({S^c})
\end{align*}
where we used the bound ${\xi^i_0}^2\leq N\SigN_0$. For large $N$, the
probability of the event $S^c$ exponentially decays with $N$ (Chernoff
bound).  As a result, for large $N$, the bound is obtained in
terms of the Lipschitz constant:
\begin{align*}
\Expect[|g(\SigN_0,t)-g(\Sigma_0,t)|^2{\xi^i_0}^2]
&\leq 
L_g^2\Expect\left[|\SigN_0-\Sigma_0|^2|{\xi}^i_0|^2\right]\\
&\leq L_g^2 \Expect\left[|\SigN_0-\Sigma_0|^4\right]^{\frac{1}{2}}\Expect\left[|{\xi}^i_0|^4\right]^{\frac{1}{2}}\\
&\leq L_g^2\frac{\Expect[|X_0-m_0|^4]^{\frac{3}{2}}}{{N}}+O(N^{-2})
\end{align*}
where we used the Lipschitz property in the first step, the Cauchy-Schwarz inequality in the second step, and $\Expect[|\SigN_0-\Sigma_0|^4]^{\frac{1}{4}}= \frac{\Expect[|X_0-m_0|^4]^\frac{1}{2}}{\sqrt{N}}+O(N^{-1})$ in the last step. 
\end{proof}

\medskip

\begin{proof}[Proof of the Corollary~\ref{cor:prop-chaos}]
Using the triangle inequality,
\begin{equation*}
\begin{aligned}
\Expect[|\frac{1}{N}\sum_{i=1}^N f(X^i_t) -\hat{f}|^2]^{1/2} \leq&\Expect[|\frac{1}{N}\sum_{i=1}^N f(X^i_t) -\frac{1}{N}\sum_{i=1}^Nf(\bar{X}^i_t)|^2]^{1/2} \\&+ \Expect[|\frac{1}{N}\sum_{i=1}^N f(\bar{X}^i_t) -\hat{f}|^2]^{1/2} 
\end{aligned}
\end{equation*}
where $\hat{f}:=\Expect[f(X_t)|\clZ_t]$. The second term is given by
\begin{equation*}
\Expect[|\frac{1}{N}\sum_{i=1}^N f(\bar{X}^i_t) -\hat{f}|^2]^{1/2} = \frac{\var(f)}{\sqrt{N}}
\end{equation*}
because $\bar{X}^i_t$ are i.i.d with distribution equal to the conditional distribution. 
It only remains to bound the first term: 
\begin{align*}
\Expect[|\frac{1}{N}\sum_{i=1}^N f(X^i_t)& -\frac{1}{N}\sum_{i=1}^N f(\bar{X}^i_t)|^2] \leq \frac{1}{N}\sum_{i=1}^N \Expect[|f(X^i_t)-f(\bar{X}^i_t)|^2]\\
&\leq \frac{\text{(const.)}}{N}\sum_{i=1}^N
\Expect[|X^i_t-\bar{X}^i_t]|^2]
\leq \frac{\text{(const.)}}{N}
\end{align*}
where we used Jensen's inequality in the first step, the Lipschitz property of $f$ in the second step, and the estimate~\eqref{eq:estimate_POA} in the last step. 
\end{proof}
\end{document}